\documentclass[12pt,a4paper]{amsart}
\usepackage{a4wide,amsfonts,amsthm,amsmath,latexsym,amssymb,euscript,graphicx,units,mathrsfs}

\usepackage[english]{babel}
\usepackage{graphicx}
\usepackage{color}
\usepackage{amssymb}
\usepackage{amssymb}
\usepackage[T1]{fontenc}
\usepackage{latexsym}
\usepackage{xypic}
\usepackage{eufrak}
\usepackage{euscript}
\usepackage{amsfonts,amsmath}
\usepackage{verbatim}
\usepackage{fancyhdr}
\usepackage{mathrsfs}
\usepackage{units}

\usepackage{hyperref}

\newtheorem{proposition}{Proposition}[section]

\newtheorem{remark}[proposition]{Remark}

\newtheorem{theorem}[proposition]{Theorem}

\newtheorem{definition}[proposition]{Definition}

\renewcommand{\geq}{\geqslant}
\def\leq{\leqslant}

\newcommand{\N}{\mathbb{N}}

\newcommand{\R}{\mathbb{R}}

\usepackage{color}

\newcommand{\on}[1]{\operatorname{#1}}

\def\1{{\mathbf{1}}}

\def\1{{\mathbf{1}}}
\def\0.5{{\frac{1}{2}}}

\newcommand{\LL}{\operatorname{L}}
\newcommand{\Id}{\operatorname{Id}}
\newcommand{\diff}[1]{\operatorname{d}\ifthenelse{\equal{#1}{}}{\,}{#1}}

\title{\bf Multivariate Gaussian approximations on Markov chaoses} 

\author{Simon Campese}
\thanks{S. Campese was partially supported by the ERC grant 277742 PASCAL}
\address[S. Campese]{Dipartimento di Matematica, Universit\`a di Roma Tor
  Vergata, Rome, Italy}
\author{Ivan Nourdin}
\author{Giovanni Peccati}
\address[I. Nourdin and G. Peccati]{Unit\'e de Recherche en Math\'ematiques,
  Universit\'e du Luxembourg, Luxembourg}
\author{Guillaume Poly}{}
\address[G. Poly]{Institut de Recherche Math\'ematique de Rennes, Universit\'e
  de Rennes 1, Rennes, France}

\email[S. Campese]{campese@mat.uniroma2.it}
\email[I. Nourdin]{ivan.nourdin@uni.lu}
\email[G. Peccati]{giovanni.peccati@gmail.com}
\email[G. Poly]{guillaume.poly@univ-rennes1.fr}


\begin{document}
\maketitle
\begin{abstract}
We prove a version of the multidimensional Fourth Moment Theorem for chaotic
random vectors, in the general context of diffusion Markov generators. In
addition to the usual componentwise convergence and unlike the
infinite-dimensional Ornstein-Uhlenbeck generator case, another moment-type
condition is required to imply joint convergence of of a given
  sequence of vectors. 
\end{abstract}

\section{Introduction}
 
The {\it Fourth Moment Theorem} (discovered by Nualart and Peccati
in~\cite{nualart_central_2005} and later extended by Nualart and
Ortiz-Latorre in~\cite{nualart_central_2008}) states that, inside a fixed Wiener 
chaos, a sequence of random variables $F_n$, $n\geq 1$, converges in
distribution towards a 
standard Gaussian random variable if and only if $\mathbb{E}[F_n^2] \to 1$ and
$\mathbb{E}[F_n^4] \to 3$. Recently, in the pathbreaking
contribution~\cite{ledoux_chaos_2012}, Ledoux approached this Fourth 
Moment Phenomenon in the more general context of diffusion Markov generators,
and was able to provide a new proof of such a result adopting a purely spectral
point of view. Later on, in~\cite{azmoodeh_fourth_2014}, Azmoodeh, Campese and
Poly generalized the concept of chaosoriginally introduced in
  \cite{ledoux_chaos_2012} and were able not only to obtain a more transparent
proof of the classical Fourth Moment Theorem, but also to exhibit many new
situations where the Fourth Moment Phenomenon occurs (e.g. the Laguerre or
Jacobi chaoses). One should notice that the collection of techniques introduced
in \cite{azmoodeh_fourth_2014} have also been successfully applied in other
contexts, e.g. for deducing moment conditions in limit theorems
(see~\cite{azmoodeh_malicet_mijoule_poly_2014}), or in the study of the
so-called {\it real Gaussian product conjecture}  (see~\cite{mnpp}).  

In this paper, we investigate a multidimensional counterpart of the
  Fourth Moment Theorem by using the aforementioned approach based on
  Markov semigroup. In the case of Wiener chaos, the multidimensional version
of the Fourth Moment Theorem is due to Peccati and Tudor
\cite{peccati_gaussian_2005}, and is given by the following
  statement. For the rest of the paper, the symbol `$\xrightarrow{d}$' indicates
  convergence in distribution. 
\begin{theorem}[See \cite{peccati_gaussian_2005}]\label{t:pt}
Let $p_1,...,p_d\geq 1$ be fixed integers and $F_n= \left(
  F_{1,n},\dots,F_{d,n} \right)$, $n\geq1$, be a sequence of vectors
such  
  that $F_{i,n}$ belongs to the $p_i$th Wiener chaos of some Gaussian field, for
  all $i=1,\ldots,d$ and all 
$n$. Furthermore, assume that 
$\lim_{n \to \infty} \on{Cov} F_n = C$, 
and denote
by 
$Z=(Z_1,\dots,Z_d)$ a centered Gaussian random vector with covariance matrix
$C$. Then, the following two assertions are equivalent, as $n\to \infty$:
\begin{enumerate}
\item[(i)] $F_n \xrightarrow{d} Z$
\item[(ii)] $F_{i,n} \xrightarrow{d} Z_i$ for all $1 \leq i \leq d$.
\end{enumerate}
\end{theorem}

In other words, for sequences of random vectors living inside a fixed system of
Wiener chaoses, {\it componentwise} 
convergence implies {\it joint} convergence. Our main 
result is the following analogue of Theorem~\ref{thm:2} in the abstract Markov  
generator framework (unexplained notation and definitions -- in
  particular the notion of a {\it chaotic vector} -- will be formally
introduced in the sequel). 

\begin{theorem}
  \label{thm:2}
  Let $\LL$ be a diffusion Markov generator with discrete spectrum $0<
  \lambda_0< \lambda_1 <\cdots$, fix integers $k_1,...,k_d\geq 1$, and let 
  $F_n=(F_{1,n},\dots,F_{d,n})$, $n \geq 1$, be a sequence of chaotic
  vectors such that $F_{i,n} \in {\rm ker}(\LL + \lambda_{k_i} \Id)$, for $1 \leq i
  \leq d$ and all $n$. Furthermore, assume that $\lim_{n \to
    \infty} 
  \on{Cov} F_n = C$ and denote by $Z=(Z_1,\dots,Z_d)$ a centered
  Gaussian random vector with covariance matrix $C$ (defined on some probability
  space $(\Omega, \mathscr{F}, \mathbb{P})$). Consider the following
  asymptotic relations {\rm (i)} and {\rm (ii)}, for $n\to \infty$: 
  \begin{enumerate}
  \item[(i)] $F_n \xrightarrow{d} Z$,
  \item[(ii)] It holds that
  \begin{enumerate}
  \item[ ${\rm (ii)}_a$] for every $i=1,...,d$, $F_{i,n} \xrightarrow{d} Z_i$,
    and  
   \item[${\rm (ii)}_b$] for every $1\leq i,j \leq d$, $$\int_E^{}  F_{i,n}^2
     F_{j,n}^2 \diff{\mu} \xrightarrow{}  \mathbb{E}[Z_i^2 
        Z_j^2].$$
    \end{enumerate}
  \end{enumerate}
  Then, {\rm (ii)} implies {\rm (i)}, and the converse implication {\rm (i)}
  $\Rightarrow$ {\rm (ii)} holds whenever the sequence $\{ F_{i,n}^2 F_{j,n}^2 :
  n\geq 1\}$ is uniformly integrable for every $1\leq i,j \leq d$. 
\end{theorem}

\begin{remark}{\rm The additional mixed moment condition ${\rm (ii)}_b$ has no counterpart in the statement of Theorem \ref{t:pt}. In Section \ref{s-main}, we will explain in detail why such a relation is automatically 
satisfied whenever the components of the vectors $F_n$ belong to the Wiener chaos of some Gaussian field. We also observe that a sufficient condition, in order for the class 
$$
F_{i,n}^2 F_{j,n}^2,\quad n\geq 1,
$$
to be uniformly integrable for every $i,j$, is that, for some $\epsilon>0$,
\begin{equation}\label{e:lulz}
\sup_n \int_E | F_{i,n} |^{4+\epsilon} \diff{\mu} <\infty, \quad \mbox{for every}\,\, i=1,\ldots, d.
\end{equation}
Finally, if the sequence $F_n$, $n\geq 1$, lives in a fixed sum of Gaussian Wiener chaoses, then \eqref{e:lulz} is automatically implied by the relation $\lim_{n \to
    \infty} 
  \on{Cov} F_n = C$, by virtue of a standard hypercontractivity argument -- see e.g. \cite[Section 2.8.3]{nourdin_normal_2012}. General sufficient conditions in order for the semigroup associated with ${\rm L}$ to be hypercontactive can be found e.g. in \cite{b}.
}
\end{remark}

As in the one-dimensional case, it is possible in our abstract framework to provide
a proof of Theorem~\ref{thm:2} that is {\it not} based on the use of product formulae, and that exploits instead the spectral information embedded into the underlying generator ${\rm L}$.

\smallskip

The rest of this paper is organized as follows. In
Section~\ref{s-preliminaries}, we will introduce the abstract Markov generator
setting and recall the main one-dimensional findings
from~\cite{azmoodeh_fourth_2014}. In 
Section~\ref{s-main}, we will define multidimensional chaos and present the
proof of Theorem~\ref{thm:2}; we also provide a careful analysis of the additional
condition  ${\rm (ii)}_b$ appearing in our Theorem \ref{thm:2}. 

\section{Preliminaries}
\label{s-preliminaries}

In this section, we introduce the general diffusion Markov generator
setting. For a detailed treatment, we refer to the
monograph~\cite{bakry_analysis_2014}.

Throughout the rest of the paper, we fix a probability space
$(E,\mathcal{F},\mu)$ and a symmetric Markov generator $\LL$ with state space
$E$ and invariant measure $\mu$. We assume that $\LL$ has discrete spectrum $S =
\{-\lambda_k : k\geq 0\}$ and order its eigenvalues by magnitude,
i.e. $0=\lambda_0 <  \lambda_1 < \lambda_2 < \dots$. In the language of
functional analysis, $\LL$ is a self-adjoint, linear operator on $L^2(E,\mu)$
with the property that $\LL 1 = 0$. By standard spectral theory, $\LL$ is
diagonalizable and we have that
\begin{equation*}
  L^2(E,\mu) = \bigoplus_{k=0}^{\infty} {\rm ker}(\LL + \lambda_k \Id).
\end{equation*}
  We denote by $\LL^{-1}$  the pseudo-inverse of $\LL$, defined on $L^2(E,\mu)$
  by   $L^{-1}1=0$ and $L^{-1}F=-\frac{1}{\lambda}F$ for  any $F\in{\rm
    ker}(\LL+\lambda\Id)$ such that $\lambda\neq 0$. It is immediate to check
  that $\LL \LL^{-1}F= F-\int_E F\diff{\mu}$ for every $F\in L^2(E,\mu)$.  
The associated bilinear carr\'e du champ operator $\Gamma$ is defined by 
\begin{equation*}
  \Gamma(F,G) = \frac{1}{2} \left( \LL (FG) - F \LL G - G \LL F \right),
\end{equation*}
whenever the right-hand side exists. As $\LL$ is self-adjoint and $\LL 1
=0$, we immediately deduce the integration by parts formula
\begin{equation}
  \label{eq:6}
  \int_E^{} \Gamma(F,G) \diff{\mu} = - \int_E^{} F \LL G \diff{\mu} = -
  \int_E^{} G \LL F \diff{\mu}.
\end{equation}
A symmetric Markov generator is called {\it diffusive}, if it satisfies the diffusion property
\begin{equation*}
  \LL \phi(F) = \phi'(F) \LL F + \phi''(F) \Gamma(F,F) 
\end{equation*}
for all smooth test functions $\phi \colon\R\to\R$ and any $F \in L^2(E,\mu)$. Equivalently,
$\Gamma$ is a derivation i.e.
\begin{equation}
  \label{eq:5}
  \Gamma(\phi(F),G) = \phi'(F) \Gamma(F,G).
\end{equation}
Considering vectors $F=(F_1,\dots,F_d)$ and test functions $\varphi \colon \R^d
\to \R$, 
the diffusion and derivation properties yield that
\begin{equation*}
  \LL \phi(F)= \LL \phi(F_1,\dots,F_d) = \sum_{i=1}^d \frac{\partial
    \varphi}{\partial x_i} (F) \, \LL F_i + \sum_{i,j=1}^d
  \frac{\partial^2}{\partial x_i \partial x_j} (F) \, \Gamma(F_i,F_j)
\end{equation*}
and
\begin{equation*}
  \Gamma(\varphi(F),G) = \sum_{i=1}^d \frac{\partial \varphi}{\partial x_i} (F)
  \, \Gamma(F_i,G),
\end{equation*}
respectively.
In~\cite{azmoodeh_fourth_2014}, the following definition of chaos was given.

\begin{definition}[See \cite{azmoodeh_fourth_2014}]
  \label{def:1}
  Let $\LL$ be a symmetric Markov generator with discrete spectrum $S$, and let
  $F \in {\rm ker}(\LL + \lambda_p \Id)$ be an eigenfunction of $\LL$ (with
  eigenvalue 
  $\lambda_p$). We say that $F$ is \emph{chaotic}, or a \emph{chaos
    eigenfunction}, if 
  \begin{equation}\label{chaos}
    F^2\,  \in \bigoplus_{\substack{k:\,\lambda_k \leq 2\lambda_p}} {\rm ker}(\LL +
    \lambda_k \Id). 
    \end{equation}
\end{definition}

Condition (\ref{chaos}) means that, if $F$ is a chaos eigenfunction of $\LL$
with eigenvalue $\lambda_p$ (say), then the (orthogonal) decomposition of its
square along the spectrum of ${\rm L}$ only contains eigenfunctions associated
with eigenvalues that are {\it less than or equal to} twice $\lambda_p$. 
Note that this property is satisfied by {\it all} eigenfunctions of ${\rm L}$ in
many crucial instances, e.g. when ${\rm L}$ is the generator of the
Ornstein-Uhlenbeck semigroup, the Laguerre generator or the Jacobi generator
(see \cite{azmoodeh_fourth_2014}). 
Starting from this definition, and by only using the spectral information
embedded into the generator $\LL$, the authors of~\cite{azmoodeh_fourth_2014}
were able to deduce Fourth Moment Theorems for sequences of chaotic
eigenfunctions and several target distributions, drastically simplifying all
known proofs. The analogue in this framework of the classical Fourth Moment Theorem
reads as follows. 

\begin{theorem}[Abstract Fourth Moment Theorem, see~\cite{azmoodeh_fourth_2014}]
  \label{thm:4}
  Let $\LL$ be a symmetric diffusion Markov generator with discrete spectrum $S$
  and let $\{ F_n : n \geq 1\}$ be a sequence of chaotic eigenfunctions of $\LL$
  with respect to the same (fixed) eigenvalue, such that $\int_E^{}
    F_n^2 \diff{\mu}  \to 1$, $n\to \infty$. Consider the following
    three conditions, as  $n\to\infty$: 
  \begin{enumerate}
  \item[(i)] $F_n \xrightarrow{d} Z$, where $Z \sim \mathcal{N}(0,1)$,
  \item[(ii)] $\int_E^{} F_n^4 \diff{\mu}  \to 3$,
  \item[(iii)] $\on{Var}(\Gamma(F_n)) \to 0$.
  \end{enumerate}
  Then, the implications {\rm (ii) } $\Rightarrow$  {\rm (iii) } $\Rightarrow$
  {\rm (i)} hold. Furthermore, if the sequence $\{ F_n^4 : n\geq 1\}$ is
  uniformly integrable, one has that  {\rm (i) } $\Rightarrow$ {\rm (ii)}.   
\end{theorem}

The next section is devoted to the proof of our main results.

\section{Main results}
\label{s-main}

Let us begin by noting the general fact that (under standard regularity
assumptions) the distance between the distribution of a random vector $F = 
(F_1,\dots,F_d)$ and a multivariate Gaussian law is controlled by the expression
$\sum_{i,j=1}^d \on{Var} \Gamma(F_i,-L^{-1} F_j)$. This fact can either be shown
by using the characteristic function method (see~\cite{nualart_central_2008}),
quantitatively by Stein's method (see~\cite{nourdin_multivariate_2010}) or by
means of the so-called ``smart path'' technique
(see~\cite{nourdin_steins_2010}). The proofs carry over to our setting almost
verbatim, by replacing the integration by parts formula of Malliavin 
calculus with the analogous relation~\eqref{eq:6} for the carr\'e du champ
operator $\Gamma$.

\smallskip

In order to keep our paper as self-contained as possible, instead of using the
above mentioned bounds, in the sequel we shall exploit the following estimate
involving characteristic functions: the proof is a Fourier-type variation of the
smart path method. 

\begin{proposition}
\label{prop:12}
Let $F =(F_1,\dots,F_d)$ be such that $F_j\in L^2(E,\mu)$ and $\int_E^{} F_j
\diff{\mu} = 0$ for $1 \leq j \leq d$. To avoid technicalities, assume
furthermore the existence of a finite $N\geq 1$ such that $F_j\in
\bigoplus_{k=0}^N 
{\rm ker}(\LL + \lambda_k 
      \Id)$ for all $1\leq j\leq d$.
Let $\gamma_d$ be the law of a $d$-dimensional
centered Gaussian random variable with covariance matrix $C$. Then,
for any $t\in\R^d$  it holds that
\begin{equation*}
  \left|
    \int_E^{} e^{i\langle t,F\rangle_2} \diff{\mu} - \int_{\R^d}^{} e^{i\langle
      t,x\rangle_2}  
    \diff\!{\gamma_d(x)} 
  \right|
  \leq
  \|t\|_2^2
  \sqrt{
    \sum_{i,j=1}^d \int_E^{} \left( C_{ij} - \Gamma(F_i,-\LL^{-1} F_j) \right)^2
    \diff{\mu} 
  }.
\end{equation*}
\end{proposition}

\begin{proof}
  Fix $t\in \R^d$ and define $\Psi(\theta) = 
e^{\frac{\theta^2}2 t^*Ct}  
  \int_E^{} e^{i\theta\langle t,F\rangle_2} \diff{\mu}$ for $\theta \in
  [0,1]$. It is straightforward to verify that $\Psi$ is differentiable on
  $(0,1)$, so that 
  \begin{equation*}
   \int_E^{} e^{i\langle t,F\rangle_2} \diff{\mu} - \int_{\R^d}^{} e^{i\langle
     t,x\rangle_2}  
    \diff\!{\gamma_d(x)} 
    =e^{-\frac{1}2 t^*Ct}  \big(
    \Psi(1) - \Psi(0) \big)=   e^{-\frac{1}2 t^*Ct}  \int_0^1  \Psi'(\theta)
    \diff{\theta}, 
  \end{equation*}
  and that the derivative of $\Psi$ is given by
  \begin{equation}
    \label{eq:106}
    \Psi'(\theta) = 
e^{\frac{\theta^2}2 t^*Ct}    
\left(\theta\,t^*Ct
    \int_E^{} e^{i\theta\langle t,F\rangle_2} \diff{\mu}
+i   \int_E^{}  \langle t,F\rangle_2\,\, e^{i\theta\langle t,F\rangle_2}
\diff{\mu}  
    \right).
  \end{equation}
Using both the integration by parts formula~\eqref{eq:6} and the diffusion
property (\ref{eq:5})  for 
  $\Gamma$ yields that  
  \begin{equation}
    \label{eq:103}
     \Psi'(\theta) = \theta
e^{\frac{\theta^2}2 t^*Ct}    
\sum_{i,j=1}^d t_it_j
    \int_E^{} ( C_{ij} - \Gamma(F_i,-\LL^{-1} F_j))\,e^{i\theta\langle
      t,F\rangle_2} \diff{\mu}. 
  \end{equation}
 The conclusion follows from an application of the Cauchy-Schwarz inequality.
\end{proof}

\begin{definition}[Jointly chaotic eigenfunction]
  \label{def:2}
  \hfill
  \begin{enumerate}
  \item[1.] Let $F_i \in {\rm ker}(\LL + \lambda_{k_i}\Id)$ and $F_j \in {\rm ker}(\LL +
    \lambda_{k_j}\Id)$ be two eigenfunctions of $\LL$. We say that $F_i$ and $F_j$
  are \emph{jointly chaotic}, if
  \begin{equation}
    F_i F_j \in \bigoplus_{r:\,\lambda_r \leq \lambda_{k_i}+\lambda_{k_j}} {\rm ker}(L +
      \lambda_r \Id).
  \end{equation}
\item[2.]  Let $F=(F_1,\dots,F_d)$ be a vector of eigenfunctions of $\LL$, such that
  $F_i \in \ker(\LL + \lambda_{k_i} \Id)$ for $1 \leq i \leq
  d$. Whenever any two of its components (possibly the same) are
      jointly chaotic, we say that $F$ 
  is \emph{chaotic}. In
  particular, this implies that each component of a chaotic vector is chaotic in
  the sense of Definition~\ref{def:1}. 
  \end{enumerate}
\end{definition}

We observe that, in many important examples, all vectors of eigenfunctions are
chaotic. In particular, this is the case for Wiener, Laguerre and Jacobi
chaos (see~\cite{azmoodeh_fourth_2014}). A crucial ingredient for the proof of
our main result is the following result, whose short proof can be found
in~\cite{azmoodeh_fourth_2014}. 

\begin{theorem}[{\cite[Thm. 2.1]{azmoodeh_fourth_2014}}]
  \label{thm:3}
  Fix an eigenvalue $-\lambda = -\lambda_n \in S$ and assume that 
  $F \in \bigoplus_{k=0}^n \on{ker} \left( \LL + \lambda_k \Id \right)$. Then it
  holds that, for any $\eta \geq \lambda_n$,
  \begin{equation*}
  \int_E^{} F \left( \LL + \eta \Id \right)^2 (F) \diff{\mu}
    \leq
    \eta
    \int_E^{} F \left( \LL + \eta \Id \right) (F) \diff{\mu}.
  \end{equation*}
\end{theorem}
We are now ready to prove our main Theorem \ref{thm:2}.

%

\begin{proof}[Proof of Theorem \ref{thm:2}] Since the remaining parts of the
  statement are straightforward, we will only prove the implication (ii)
  $\Rightarrow$ (i). According to Proposition \ref{prop:12}, it 
  suffices to show that, if (ii) is satisfied then, for $1 \leq i,j \leq d$, 
  \begin{equation*}
    \int_E^{} \left( \Gamma \left( F_{i,n}, -\LL^{-1} F_{j,n} \right) - C_{ij}
    \right)^2 \diff{\mu} \to 0,
  \end{equation*}
  as $n\to\infty$. For $i=j$, this follows from the one-dimensional Fourth Moment
  Theorem~\ref{thm:4}. Let us thus assume that $i \neq j$. For the sake of readability, we
  temporarily suppress the index $n$ from $F_{i,n}$ and $F_{j,n}$.
  It holds that $\Gamma(F_i, -\LL^{-1} F_j) - C_{ij} =
  \frac{1}{\lambda_{k_{j}}}  \Gamma(F_i,F_j) - C_{ij} $ and, by 
  definition of $\Gamma$, we can write
  \begin{equation*}
    \Gamma(F_i, F_j) - \lambda_{k_j} C_{ij} = \frac{1}{2} \left( \LL +
      (\lambda_{k_i}+ \lambda_{k_j}) \Id \right) \left( F_iF_j - a_{ij}
       C_{ij}\right),
   \end{equation*}
   where $a_{ij} = \frac{2\lambda_{k_j}}{\lambda_{k_i}+\lambda_{k_j}}$.
   Inserting the definition of the carr\'e du champ operator, we get that
  \begin{align}
     \int_E^{} \big( \Gamma(F_i,&-\LL^{-1} F_j) - C_{ij} \big)^2
    \diff{\mu}  \notag
    \\ &= \notag
         \frac{1}{\lambda_{k_j}^2}
    \int_E^{} \left( \Gamma(F_i, F_j) - \lambda_{k_j} C_{ij} \right)^2 \diff{\mu}
    \\ &= \notag
         \frac{1}{4\lambda_{k_j}^2}
         \int_E^{}
         \left( 
      \left( \LL +
      (\lambda_{k_i}+ \lambda_{k_j}) \Id \right) \left( F_iF_j - a_{ij} C_{ij}
      \right) \right)^2 \diff{\mu}
         \\ &= \notag
         \frac{1}{4\lambda_{k_j}^2}
         \int_E^{}
         \left(
         F_iF_j - a_{ij} C_{ij} \right)
      \left( \LL +
      (\lambda_{k_i}+ \lambda_{k_j}) \Id \right)^2 \left( F_iF_j - a_{ij} C_{ij}
              \right)  \diff{\mu}
  \end{align}
  Therefore, by Theorem~\ref{thm:3},
  \begin{align}
  \int_E^{} \big( \Gamma(F_i,&-\LL^{-1} F_j) - C_{ij} \big)^2
    \diff{\mu}  \notag    
   \\ &\leq \notag
      \frac{ \lambda_{k_i} +\lambda_{k_j}}{4\lambda_{k_j}^2}
         \int_E^{}
         \left( F_iF_j - a_{ij} C_{ij}\right) 
      \left( \LL +
      (\lambda_{k_i}+ \lambda_{k_j}) \Id \right) \left( F_iF_j - a_{ij}
         C_{ij}\right) \diff{\mu}
    \\ &= \notag
         \frac{1}{a_{ij} \lambda_{k_j}}
         \int_E^{}
           \left( F_iF_j - a_{ij} C_{ij}\right)
           \left( \Gamma(F_i,F_j) - \lambda_{k_j} C_{ij} \right)
         \diff{\mu}
    \\ &=
         \frac{1}{a_{ij} \lambda_{k_j}}
         \left(
           \int_E^{} F_i F_j \Gamma(F_i,F_j) \diff{\mu} 
           - a_{ij} C_{ij} \lambda_{k_j} \int_E^{} F_iF_j
         \diff{\mu}\right.  \label{eq:2} 
    \\
      &\left. 
          \quad \quad \quad\quad \quad \quad\quad \quad \quad \quad  \quad \quad \quad \quad \quad \quad- \lambda_{k_j} C_{ij} \int_E^{} \left( F_iF_j -  a_{ij} C_{ij}\right)  \diff{\mu}
         \right).\notag         
  \end{align}
  Now, by the diffusion property~\eqref{eq:5} and the integration by parts
  formula~\eqref{eq:6}, we  have
  \begin{eqnarray*}
    \int_E^{} F_i F_j \Gamma(F_i,F_j) \diff{\mu}
    &=&
      \frac{1}{4}
      \int_E^{} \Gamma(F_i^2,F_j^2) \diff{\mu}=
         - \frac{1}{4}
         \int_E^{} F_i^2 \LL (F_j^2) \diff{\mu}
    \\ &=&
         - \frac{1}{2}
         \int_E^{} F_i^2 \left( F_j \LL F_j + \Gamma(F_j,F_j) \right) \diff{\mu}
    \\ &=&
         \frac{\lambda_{k_j}}{2}
         \int_E^{} F_i^2 F_j^2 \diff{\mu}
         -
         \frac{1}{2}
         \int_E^{} F_i^2 \Gamma(F_j,F_j) \diff{\mu}
    \\ &=&
         \frac{\lambda_{k_j}}{2}
         \int_E^{} F_i^2 \left( F_j^2 - C_{jj} \right)
         \diff{\mu}
         -
         \frac{1}{2} \int_E^{} F_i^2 \left(  \Gamma(F_j,F_j) - \lambda_{k_j}
         C_{jj} \right) \diff{\mu}
    \\ &\leq&
         \frac{\lambda_{k_j}}{2}
          \int_E^{} F_i^2 \left( F_j^2 - C_{jj} \right)
         \diff{\mu}
         +
         \frac{1}{2} \sqrt{m_4(F_i)}\,\sqrt{ {\rm Var}(\Gamma(F_j,F_j))}
            \\ &=&
         \frac{\lambda_{k_j}}{2} \bigg(
          \int_E^{} F_i^2 F_j^2          \diff{\mu}- C_{ii}C_{jj}
          \bigg)
         +
         \frac{1}{2} \sqrt{m_4(F_i)}\,\sqrt{ {\rm Var}(\Gamma(F_j,F_j))}.
  \end{eqnarray*}
  Plugging such an estimate into~\eqref{eq:2} and reintroducing the index $n$,  we can thus
  write
  \begin{multline*}
    \int_E^{} \left( \Gamma(F_{i,n},-\LL^{-1} F_{j,n}) - C_{ij} \right)^2
    \diff{\mu}
    \\ \leq
    \frac{1}{a_{ij} \lambda_{k_j}}
    \left(
      \frac{1}{2} \sqrt{ m_4(F_{i,n})}\,\sqrt{  {\rm Var}(\Gamma(F_{j,n},F_{j,n})) }+ R_{i,j}(n)
      \right),
    \end{multline*}
  where
\[
          R_{ij}(n)
    = \lambda_{k_j} \left( \frac{1}{2} \int_E^{} F_{i,n}^2  F_{j,n}^2  \diff{\mu}  - \frac12  C_{ii,n}C_{jj,n}
   - a_{ij} C_{ij,n}^2  \right)
       -
      \frac{C_{ij,n}^2(1-a_{ij})}{\lambda_{k_j}}.
\]
By virtue of Theorem \ref{thm:4}, we are now left to show that $R_{ij}(n)\to 0$, as $n\to\infty$. If $\lambda_{k_i} = \lambda_{k_j}$, we have that $a_{ij}=1$ and thus
\[
    R_{ij}(n) = \frac{\lambda_{k_j}}{2} \left( \int_E^{} F_{i,n}^2  F_{j,n}^2  \diff{\mu}  - C_{ii,n}C_{jj,n}
   - 2 C_{ij,n}^2  \right)
\]
  As $\mathbb{E}[Z_i^2 Z_j^2] =  \mathbb{E}[Z_i^2] \mathbb{E}[Z_j^2] + 2 \mathbb{E}[Z_{ij}]^2$, we see
  that due to assumption ${\rm (ii)}_b$, $R_{ij}(n) \to 0$ as $n \to \infty$.
  In the case where $\lambda_{k_i} \neq \lambda_{k_j}$, it necessarily holds
  that $C_{ij,n} = 0$ (by orthogonality of eigenfunctions of different orders) and thus
  \begin{equation*}
    R_{ij}(n) = \frac{\lambda_{k_j}}{2} \left( \int_E^{} F_{i,n}^2 F_{j,n}^2
      \diff{\mu}
    - C_{ii,n}C_{jj,n}
 \right),
  \end{equation*}
  which, again due to assumption ${\rm (ii)}_b$, vanishes in the limit as well.

\end{proof}

\begin{remark}
{\rm
  In the framework of the classical Theorem \ref{t:pt}, where $\LL$ is the
  infinite-dimensional Ornstein-Uhlenbeck generator and its eigenfunctions are
  multiple Wiener-It\^o integrals, condition ${\rm (ii)}$ actually reduces to
  ${\rm (ii)}_a$, since ${\rm (ii)}_a\Rightarrow {\rm (ii)}_b$. In 
  other words, componentwise Gaussian convergence in distribution always yields
  joint Gaussian convergence in this case. From our abstract point of view, we
  can explain (and generalize) this phenomenon by means of the next result.
  }
\end{remark}

\begin{proposition}
  In the setting and with the notation of Theorem~\ref{thm:2}, assume in
  addition that 
  \begin{enumerate}
  \item[1.] $\LL$ is ergodic, i.e. its kernel only consists of constant functions,
  \item[2.] For $1 \leq i,j \leq d$ such that $\lambda_{k_i}=\lambda_{k_j}$, it
    holds that 
    \begin{equation}
      \label{eq:3}
      \int_E^{} \pi_{2\lambda_{k_i}}(F_{i,n}^2) \pi_{2\lambda_{k_i}}(F_{j,n}^2)
      \diff{\mu} -  2 \left( \int_E^{} F_{i,n} F_{j,n} \diff{\mu}
      \right)^2 \xrightarrow{} 0,
    \end{equation}
    where $\pi_{\lambda}$ denotes the orthogonal projection onto ${\rm ker}(\LL +
      \lambda \Id)$.
  \end{enumerate}
  Then, the following two assertions are equivalent.
  \begin{enumerate}
  \item[(i)] $F_n \xrightarrow{d} Z$
  \item[(ii)] For $1 \leq i \leq d$ it holds that $F_{i,n} \xrightarrow{d} Z_i$.
  \end{enumerate}
\end{proposition}

\begin{proof}

  In view of Theorem \ref{thm:2}, 
  we have only to show that condition ${\rm (ii)}_a$ therein implies that
  \begin{equation*}
    \int_E^{} F_{i,n}^2 F_{j,n}^2 \diff{\mu} \to \mathbb{E}[Z_i^2 Z_j^2]
  \end{equation*}
  or equivalently (as
  $\mathbb{E}[Z_i^2Z_j^2]=\mathbb{E}[Z_i^2]\mathbb{E}[Z_j^2]+2 
  \mathbb{E}[Z_iZ_j]^2$) that 
  \begin{equation*}
        \int_E^{} F_{i,n}^2 F_{j,n}^2 \diff{\mu} - \int_E^{} F_{i,n}^2
        \diff{\mu} \int_E^{} F_{j,n}^2 \diff{\mu} - 2 \left( \int_E^{} F_{i,n}
          F_{j,n} \diff{\mu} \right)^2 \to 0
  \end{equation*}
  for $1 \leq i,j \leq d$.
  To do so, we first note that, by the definition of $\Gamma$ and the chaotic
  property of $F_{i,n}$, it holds that
  \begin{equation*}
    \Gamma(F_{i,n},F_{i,n})
    =
    \frac12\left( \LL + 2\lambda_{k_i} \Id \right) (F_{i,n}^2)
    =
    \frac12\sum_{\substack{r \in \N \\ \lambda_r < 2\lambda_{k_i}}}^{}
    (2\lambda_{k_i} - \lambda_r) \pi_{\lambda_r}(F_{i,n}^2).
  \end{equation*}
  Therefore, by orthogonality of the projections corresponding to different eigenvalues and by the one-dimensional Fourth Moment Theorem~\ref{thm:4}, $F_{i,n}
  \xrightarrow{d} Z_i$ implies that 
  \begin{equation}
    \label{eq:4}
    \int_E^{} \pi_{\lambda_r}(F_{i,n}^{2})^2 \diff{\mu} \to  0, \qquad\mbox{for all $r$ such that }
    0 < \lambda_r < 2\lambda_{k_i}.
  \end{equation}
  We exploit this fact by writing
  \begin{align*}
    \int_E^{} F_{i,n}^2 F_{j,n}^2 \diff{\mu}
    &=
      \int_E^{} \pi_0 \left( F_{i,n}^2 \right) F_{j,n}^2 \diff{\mu}
     +
      \sum_{\substack{r \in \N \\ 0 < \lambda_r < 2\lambda_{k_i}}}
      \int_E^{} \pi_{\lambda_r} \left( F_{i,n}^2 \right) F_{j,n}^2 \diff{\mu}
      \\ &\qquad \qquad  \qquad \qquad \qquad+
      \int_E^{} \pi_{2\lambda_{k_i}} \left( F_{i,n}^2 \right) F_{j,n}^2
      \diff{\mu} .
  \end{align*}
Assume without loss of generality that $\lambda_{k_j} \leq \lambda_{k_i}$.
  The ergodicity assumption on $\LL$ forces $\pi_0(F_{i,n}^2)$ to be constant
  and thus 
\begin{equation*}
  \int_E^{} \pi_0(F_{i,n}^2) F_{j,n}^2 \diff{\mu}
  =
  \pi_0(F_{i,n}^2) \int_E^{} F_{j,n}^2 \diff{\mu}
  =
  \int_E^{} F_{i,n}^2 \diff{\mu} \int_E^{} F_{j,n}^2 \diff{\mu}.
\end{equation*}
By Cauchy-Schwarz and~\eqref{eq:4}, all integrals $\int_E^{}
\pi_{\lambda_r}(F_{i,n}^2) F_{j,n}^2 \diff{\mu}$ inside the sum in the middle
vanish in the limit. Finally, assumption~\eqref{eq:3} ensures that the third
term (which is zero if $\lambda_{k_i} \neq \lambda_{k_j}$), exhibits the wanted
asymptotic behaviour. 
\end{proof}

\begin{remark}{\rm
  As already mentioned, both of these additional assumptions are always verified
  in the case of the infinite-dimensional Ornstein-Uhlenbeck generator
  (see for example~\cite{nourdin_normal_2012} or~\cite{nualart_malliavin_2006}
  for any unexplained notation). While the ergodicity is immediate, more effort
  is needed to show~\eqref{eq:3}. For a 
  multiple integral $I_p(f)$ with $f \in \mathfrak{H}^{\odot p}$, the well known
  product formula yields that 
  $\pi_{2p}(I_p(f)^2)= I_{2p}(f \widetilde{\otimes} f)$. If $I_p(g)$, $g \in
  \mathfrak{H}^{\odot p}$ is another multiple integral, one can show 
  (see for example~\cite[Lemma 2.2(2)]{nourdin_asymptotic_2014}) that
  \begin{equation*}
    \mathbb{E}[I_{2p}(f \widetilde{\otimes} f) I_{2p} (g \widetilde{\otimes} g)]
  =
         2 \mathbb{E}[I_p(f) I_p(g)]^2 +
         \sum_{r=1}^{p-1} p!^2 \binom{p}{r}^2 \left\langle f \otimes_r f, g
         \otimes_r g \right\rangle_{\mathfrak{H}^{\otimes 2p}}.
  \end{equation*}
  Replacing the kernels $f$ and $g$ by two sequences $(f_n)$ and $(g_n)$, the
  classical Fourth Moment Theorem implies that scalar products inside
  the sum vanish in the limit if (at least one of) the two sequences
  $I_p(f_n)$ and $I_p(g_n)$ converges in distribution towards a Gaussian.
  }
\end{remark}


\end{document}